\documentclass[11pt]{amsart}

\usepackage{amsmath,amssymb, mathrsfs}
\usepackage[hidelinks]{hyperref}
\usepackage{enumitem}

\numberwithin{equation}{section}

\newtheorem{theorem}{Theorem}[section]

\newtheorem{lemma}[theorem]{Lemma}

\newtheorem{corollary}[theorem]{Corollary}

\theoremstyle{definition}
\newtheorem{example}[theorem]{Example}

\newtheorem{definition}[theorem]{Definition}

\DeclareMathOperator{\Tv}{TV}
\DeclareMathOperator{\Haus}{\mathscr{H}}
\DeclareMathOperator{\R}{\mathbb{R}}

\newcommand{\norm}[1]{\|#1\|}
\newcommand{\abs}[1]{\lvert#1\rvert}

\begin{document}

\title[Integral representation of functions on the circle]{Integral representation of functions \\ on the circle}
\author{Giuliano Basso}
\address{Max Planck Institute for Mathematics, Vivatsgasse 7, 53111 Bonn, Germany}
\email{giuliano.basso@web.de}

\begin{abstract}
We give a complete characterization of all real-valued functions on the unit circle \(S^1\) that can be represented by integrating the spherical distance on \(S^1\) with respect to a signed measure or a probability measure. 
\end{abstract}

\maketitle 	
		
\section{Introduction}	
Given a bounded complete metric space \((X,d)\), we let \(\mathcal{P}(X)\) denote the set of all Borel probability measures on \(X\). For \(\mu\), \(\nu\in \mathcal{P}(X)\) we say that \(\pi\in \mathcal{P}(X\times X)\) is a \textit{coupling} of \((\mu, \nu)\) if \(\pi(A\times X)=\mu(A)\) 
and \(\pi(X\times A)=\nu(A)\) for all Borel subsets \(A\subset X\). The following expression 
\[
W_1(\mu, \nu)=\inf \int_{X\times X} d(x, y) \, \mathrm{d}\pi(x, y),
\]
where the infimum is taken over all couplings \(\pi\) of \((\mu, \nu)\), defines a metric on \(\mathcal{P}(X)\). We refer to this metric as the \(1\)-Wasserstein distance. Wasserstein distances on general metric spaces are an important object of study in optimal transport theory. In the present article, we will focus on the special case when \(X\) is the unit circle \(S^1\subset \R^2\) equipped with the spherical distance \(d_{S^1}\). Recall that the spherical distance \(d_{S^n}\) on the unit sphere \(S^n\subset \R^{n+1}\) is defined by 
\[
d_{S^n}(x,y)=\arccos( \langle x, y \rangle_{\R^{n+1}}).
\]
 In particular, for distinct points \(x\), \(y\in S^1\), we find that \(d_{S^1}(x, y)\) is equal to the length of the shorter arc of \(S^{1}\setminus\{x, y\}\). 
 
The canonical embedding \(\delta\colon S^1\to \mathcal{P}(S^1)\) defined by \(x\mapsto \delta_x\) is an isometric embedding (meaning that \(W_1(\delta(x), \delta(y))=d_{S^1}(x,y)\) for all \(x\), \(y\in S^1\)). This follows directly from the general observation that the product measure \(\delta_x \otimes \delta_y\) is the only coupling of \((\delta_x, \delta_y)\). In general, it seems difficult to predict which metric spaces besides \(S^1\) also admit an isometric embedding into \(\mathcal{P}(S^1)\). To the author's knowledge the only known result in this direction is the following quite remarkable theorem by Creutz \cite{creutz-2020}. 
 
 \begin{theorem}[Creutz embedding theorem]
The canonical embedding \(\delta\colon S^1\to \mathcal{P}(S^1)\) can be extended to an isometric embedding of the closed upper hemisphere \(H^+\subset S^2\) into \(\mathcal{P}(S^1)\).
 \end{theorem}

Such an isometric extension \(\Phi\colon H^+\to \mathcal{P}(S^1)\) can be constructed explicitly. For \(p\in H^+\setminus S^1\), the measures \(\Phi(p)\) can be taken to be absolutely continuous with respect to the normalized Hausdorff \(1\)-measure \(\Haus^1\) on \(S^1\). Let \(f_p\colon S^1\to \R\) be defined by \(f_p(x)=d_{S^2}(p, x)\) for \(p\in H^+\setminus S^1\). Creutz showed that there exists a density \(\varrho_p\) on \(S^1\) depending only on \(f_p^{\prime\prime}\) such that 
\[
\Phi(p)=\varrho_p \Haus^1
\]
is a probability measure for which
\begin{equation}\label{eq:we-want}
f_p(x)=\int_{S^1} d_{S^1}(x,y) \, \Phi(p)(\mathrm{d}y) 
\end{equation}
for all \(x\in S^1\). Notice that the right hand side of \eqref{eq:we-want} is nothing but \(W_1(\Phi(p), \delta_x)\). Hence, in particular \(d_{S^2}(p, x)=W_1(\Phi(p), \delta_x)\). Using an analytic expression for the \(1\)-Wasserstein distance on \(S^1\) by Cabrelli and Molter \cite{cabrelli-1995}, Creutz moreover showed that
\[
W_1(\Phi(p), \Phi(q))=\norm{f_p-f_q}_{\infty}
\]
for all \(p\), \(q\in H^+\setminus S^1\). Since any two points \(p\), \(q\in H^+\setminus S^1\) lie on a great circle which intersects \(S^1\), it is now easy to check that \(\norm{f_p-f_q}_{\infty}=d_{S^2}(p, q)\). Hence, \(\Phi\) is an isometric embedding.

In this article, we are interested in the natural question which other functions \(f\colon S^1\to \R\) admit an integral representation as in \eqref{eq:we-want}. To answer this question it turns out to be beneficial to work within the more general framework of signed measures. 

\begin{definition}
A function \(f\colon S^1\to \R\) is \textit{representable by a signed measure} if there exists a signed Borel measure \(\lambda\) on \(S^1\) such that 
\begin{equation}\label{eq:rep-1}
f(x)=\int_{S^1} d_{S^1}(x,y) \, \lambda(\mathrm{d}y). 
\end{equation}
for all \(x\in S^1\).
\end{definition}

Different measures may represent the same function. For example, if \(\lambda\) satisfies \(\lambda(S^1)=1\) and \(\lambda(A)=\lambda(-A)\) for all Borel subsets \(A\subset S^1\), then we have
\[
\int_{S^1} d_{S^1}(x,y) \, \lambda(\mathrm{d}y)=\frac{1}{2}\int_{S^1} \big[d_{S^1}(x,y)+d_{S^1}(x, -y)\big] \, \lambda(\mathrm{d}y)=\frac{\pi}{2}
\]
and so any such \(\lambda\) induces the constant function \(f\equiv \frac{\pi}{2}\). Let \(T\colon S^1\to S^1\) denote the antipodal map \(T(x)=-x\). Clearly, every signed measure \(\lambda\) 
admits a decomposition 
\[
\lambda=\lambda^a+\lambda^s
\]
with \(T_\# \lambda^a=-\lambda^a\) and \(T_\# \lambda^s=\lambda^s\). Here, we use \(T_\# \lambda\) to denote the push-forward of \(\lambda\) under \(T\). Let \(f_\lambda\) denote the right-hand side of \eqref{eq:rep-1}. By the above, it follows that
\begin{equation}\label{eq:useful-1}
f_\lambda=f_{\lambda^a}+(\pi/2)\cdot \lambda(S^1).
\end{equation}
Hence, only the anti-symmetric part of \(\lambda\) induces non-trivial integral representations. Our first result shows that the assignment \(\lambda \mapsto f_\lambda\) is injective when restricted to anti-symmetric measures.

\begin{lemma}\label{lem:uno}
If \(f\colon S^1\to \R\) is induced by \(\lambda\) and \(\eta\), then  \(\lambda^a=\eta^a\). In other words, whenever \(f\) is induced by \(\lambda\) then the anti-symmetric part \(\lambda^a\) of \(\lambda\) is unique.
\end{lemma}

We now proceed by setting the stage for our main result, Theorem~\ref{thm:main}, which provides an equivalent condition for a function \(f\) on \(S^1\) to be representable by a signed measure.

Let \(q \colon \R \to S^1\) be the covering map defined by \(q(t)=(\cos(t), \sin(t))\). This induces a natural left action of \(\R\) on \(S^1\) by setting \(x+_{q}t=q(a+t)\), for any \(a\in q^{-1}(x)\). We say that \(f\colon S^1\to \R\) is \textit{left-differentiable} (with respect to \(q\)) if  the limit 
\[
\partial_{-}f(x)=\lim_{t\to 0^-} \frac{f(x+_qt)-f(x)}{t}
\]
exists for each \(x\in S^1\). For example, for every \(p\in S^1\) the function \(d_p(x)=d_{S^1}(p,x)\) is left-differentiable. The \emph{total variation} of a function \(f\) on \(S^1\) is defined by
\[
\norm{f}_{\Tv(S^1)}=\sup\sum_{i=0}^n \abs{f(x_i)-f(x_{i+1})},
\]
where the supremum is taken over all partitions \(P=\{x_0, \ldots, x_{n+1}\}\) of \(S^1\). To be precise, \(P\) is called a partition of \(S^1\) if there exists a partition \(t_0 \leq \dotsm \leq t_{n+1}\) of \([0, 2\pi]\subset \R\) such that \(x_i=q(t_i)\). 

Our main theorem states that the above definitions already suffice to fully characterize those functions on the circle that are representable by a signed measure.

\begin{theorem}\label{thm:main}
Let \(f\colon S^1\to \R\) be a function and \(C\in \R\) a real number. Consider the following 
two conditions:\vspace{0.5em}

\begin{enumerate}[label=(\Alph*)]
\item\label{it:a} the sum of images of antipodal points equals \(\pi\cdot C\), that is, 
\[
f(x)+f(-x)=\pi\cdot C
\]
for all \(x\in S^1\). \vspace{0.5em}

\item\label{it:b} \(f\) is Lipschitz continuous and left-differentiable such that the total variation of \(\partial_{-} f\) is finite.
\end{enumerate}
\vspace{0.5em}
Then, assuming \ref{it:a} and \ref{it:b} is equivalent to the existence of a unique signed measure 
\(\lambda\) on \(S^1\)  with \(T_\#\lambda=-\lambda\) such that \(f\) is representable by \(\bar{\lambda}=\lambda+C\cdot\Haus^1\).
\end{theorem}

The measure \(\lambda\) will be a multiple of the Stieltjes measure associated to \(\partial_{-} f\). We emphasize that 
it is not clear \textit{a priori} whether \(\partial_- f\) is left continuous or not. Therefore, an important part of 
the proof will be devoted to deriving this property for \(\partial_- f\). 

There are many examples of functions satisfying the conditions of the theorem. For instance,  a natural class of functions satisfying condition \ref{it:a} are elements of the injective hull \(E(S^1)\) of \(S^1\). A metric space \(Y\) is called \textit{injective} if every \(1\)-Lipschitz map \(f\colon A \to Y\) from any subset \(A\) of a metric space \(X\) can be extended to a \(1\)-Lipschitz map \(F\colon X \to Y\). Basic examples of injective metric spaces include the real line, the Banach spaces \(\ell_\infty^n\) and complete metric \(\R\)-trees. A deep result of Isbell \cite{isbell-1964} from the 1960s shows that every metric space \(X\) has an (essentially) unique injective hull \(E(X)\). This injective metric space can be characterized as the smallest injective space containing \(X\) isometrically. In \cite{moulton-2000}, it is shown that
\[
E(S^1)=\big\{ f\colon S^1 \to \R : \text{\(f\) is \(1\)-Lipschitz and \(f(x)+f(-x)=\pi\)}\big\}
\]
equipped with the supremum norm. See also \cite{lim2021some}. Hence, functions \(f\in E(S^1)\) satisfy condition \ref{it:a} and thus \(f\in E(S^1)\) is representable by a signed measure if and only if it satisfies condition \ref{it:b} of Theorem~\ref{thm:main}.

We now deal with finite Borel measures \(\mu\colon \mathcal{B}(S^1)\to [0, \infty)\) on \(S^1\). One might ask whether every function that is representable by a signed measure is also representable by such a measure. However, this is not possible in general (see Example~\ref{par:ex} below).  By closely inspecting the proof of Theorem \ref{thm:main},  we get the following characterization of functions on \(S^1\) that are representable by a non-negative measure.

\begin{corollary}\label{cor:main}
Let \(f\colon S^1\to \R\) be a function and \(C\in \R_{\geq 0}\) a non-negative real number. Then the following statements are equivalent:
\vspace{0.5em}
\begin{enumerate}[label=(\roman*)]
\item\label{it:two} \(f\) is representable by a Borel measure with total mass \(C\). \vspace{0.5em}
\item\label{it:one} \(f\) satisfies \ref{it:a} and \ref{it:b} of Theorem \ref{thm:main} and \(\norm{\partial_{-} f}_{_{\Tv(S^1)}}\leq 4C\).
\end{enumerate}
\vspace{0.5em}
Moreover, if \(f\) is representable by a measure with total mass \(C\), then there exists a unique Borel measure \(\mu\) on \(S^1\)  with \(\mu(S^1)=\frac{1}{4} \norm{\partial_{-} f}_{_{\Tv(S^1)}}\) such that 
\[
\bar{\mu}=\mu+\big[C-\tfrac{1}{4}\cdot \norm{\partial_{-} f}_{_{\Tv(S^1)}}\big]\cdot\Haus^1
\]
is a non-negative measure and \(f\) is representable by \(\bar{\mu}\).
\end{corollary}

The appearance of the factor \(4\) in \(\ref{it:one}\) may seem surprising at first sight. It can be interpreted as follows. For \(p\in S^1\) let \(d_p\colon S^1 \to \R\) be defined by \(d_p(x)=d_{S^1}(p, x)\). Then \(d_p \circ q\) is a translate of the 'zigzag' function \(z\colon \R \to \R\) which is the unique \(2\pi\)-periodic function such that \(z(t)=\abs{t}\) on \([-\pi, \pi)\). Thus, the left derivative of \(d_p \circ q\) alternates periodically between the values \(-1\) and \(1\), and its restriction to \([-\pi, \pi)\) has exactly two break points. In particular, we find that the total variation of \(\partial_{-} d_p\) is equal to \(4\). Clearly, \(d_p\) is representable by a probability measure (take \(\mu=\delta_p\)). Now, Corollary~\ref{cor:main} tells us that for \(f\) to be representable by a probability measure it is necessary that the total variation of \(\partial_- f\) is no greater than the total variation of \(\partial_- d_p\).

\section{Proof of the main results}\label{sec:proofOfMain}

Let \(X=(X,d)\) denote a metric space and \(\mathcal{B}(X)\) its Borel \(\sigma\)-algebra. 
A function \(\lambda\colon \mathcal{B}(X)\to \R\) is called a \textit{signed measure} if \(\lambda(\emptyset)=0\) and \(\lambda\) is countably additive, that is,
\[
\lambda(A)=\sum_{i=1}^{\infty} \lambda(A_i)
\]
for all \(A\in \mathcal{B}(X)\) and all countable Borel partitions \((A_i)\) of \(A\). 
The proof of Theorem~\ref{thm:main} makes heavy use of the following well-known consequence of Fubini's theorem.
\begin{lemma}\label{lem:fubini}
Let \(T\geq 0\) be a real number and \(g\colon S^1\to [-T, T]\) a Borel measurable function. Further, 
let \(\phi\colon [-T,T]\to \R\) be absolutely continuous.
If \(\lambda\) is a signed measure on \(S^1\), then
\begin{equation}\label{eq:Fubini}
\int_{S^1} \phi(g(x)) \, \lambda(\mathrm{d}x)=\phi(T) \lambda(S^1)-\int\limits_{-T}^{T} \phi^\prime(t) \,\lambda\big( \{x\in S^1 : g(x) < t \}\big)\, \mathrm{d}t.
\end{equation} 
\end{lemma}
\begin{proof}
Omitted.
\end{proof}

Recall that \(\abs{\lambda}\colon \mathcal{B}(X) \to [0, \infty)\) defined by 
\begin{equation*}
B\mapsto \sup\bigg\{ \sum_{i=1}^n \abs{\lambda(B_i)} : (B_i)_{i=1}^n \text{ Borel partition of \(B\)} \bigg\}
\end{equation*}
is a Borel measure on \(X\). It is called \textit{the total variation of \(\lambda\)}.
Now, we are already in a position to prove Theorem~\ref{thm:main}.
 
\begin{proof}[Proof of Theorem~\ref{thm:main}]
We begin by showing the reverse implication. Since by assumption \(T_\#\lambda=-\lambda\), we get that \(\lambda(S^1)=0\) and thus \(\bar{\lambda}(S^1)=C\). Moreover, for all \(x\in S^1\) we compute
\[
f_{\bar{\lambda}}(x)+f_{\bar{\lambda}}(T(x))=\int_{S^1} d_{S^1}(x,y) \, \bar{\lambda}(\mathrm{d}y)+\int_{S^1}d_{S^1}(T(x),y) \, \bar{\lambda}(\mathrm{d}y)=\pi \cdot \bar{\lambda}(S^1).
\]
This establishes \ref{it:a}. Next, we show \ref{it:b}. Since each function \(d_{S^1}(\cdot, y)\) is \(1\)-Lipschitz, 
it is easy to check that \(f_{\bar{\lambda}}\) is \(L\)-Lipschitz for \(L=\abs{\lambda}(S^1)\). 
By virtue of Lebesgue's dominated convergence theorem, \(f_{\bar{\lambda}}\) is left-differentiable. Indeed, we have
\begin{equation}\label{eq:crucial-formula}
\begin{split}
\partial_-f_{\bar{\lambda}}(x)&=\int_{S^1} \partial_-(d_{S^1}(\cdot, y))(x) \, \lambda(\mathrm{d}y) \\
&=\lambda[T(x), x)-\lambda[x, T(x)),
\end{split}
\end{equation}
where \([x, T(x))=\{ x+_q t : t\in [0, \pi)\}\). Here, as in the introduction, \(q\colon \R\to S^1\) is 
defined by \(q(t)=(\cos(t), \sin(t))\). By the above, it follows that 
\begin{equation*}
\norm{\partial_{-} f_{\bar{\lambda}}}_{_{\Tv(S^1)}}\leq \abs{\lambda}(S^1).
\end{equation*}
This establishes \ref{it:b}, as desired. 

Conversely, suppose now that conditions \ref{it:a} and \ref{it:b} hold. Due to \ref{it:b} we obtain that \(h:=f\circ q\) is left-differentiable and the total variation of \(\partial_- h\) is finite on every bounded interval \(I\subset \R\). In particular, there exist bounded non-decreasing functions \(\alpha\), \(\beta\colon [-\pi, \pi]\to \R\) such that \(\partial_{-} h=\alpha-\beta\) on \([-\pi,\pi]\). Since \(h|_{[-\pi, \pi]}\) is Lipschitz, it follows that
\[
h(t)-h(-\pi)=\int\limits_{-\pi}^t \partial_- h(s)\, \mathrm{d}s=\int\limits_{-\pi}^t \alpha(s)\, \mathrm{d}s-\int\limits_{-\pi}^t \beta(s)\, \mathrm{d}s.
\]
Hence, \(h|_{[-\pi,\pi]}\) can be written as the difference of two convex functions. Since the left-derivative of a convex function is left-continuous (see e.g. \cite[Theorem 24.1]{rockafellar--1970}), this implies that \(\partial_{-}h\) is left-continuous.  

We now consider the Stieltjes measure \(\lambda^\star:=d(\partial_{-} h)\) on \([\pi, \pi)\). It is well-known that this is the unique signed Borel measure on \([-\pi, \pi)\) such that
\begin{equation}\label{eq:Stjielt}
\lambda^\star[s,t)=\partial_{-}h(t)-\partial_{-}h(s)
\end{equation}
for all \(s,t\in [-\pi, \pi)\) with \(s\leq t\). We set \(\lambda=q_\# \lambda^\star\). Because of \ref{it:a}, we find that
\begin{equation}\label{eq:antiS}
\partial_{-}h(t)=-\partial_{-}h(t-\pi)
\end{equation}
for all \(t\in \R\); as a result, \(T_\#\lambda=-\lambda\). Using Lemma \ref{lem:fubini}, we get for each \(t\in [0,\pi)\) that
\begin{equation}\label{eq:AppFub}
\begin{split}
\int_{S^1} d(q(t),y) \, \lambda(\mathrm{d}y)=&\, \pi \cdot \lambda(S^1)+\int\limits_{t-\pi}^{t}  \,\lambda^\star[-\pi, s) \, \mathrm{d}s \\
&-\int\limits_{-\pi}^{t-\pi}  \,\lambda^\star[-\pi, s)\, \mathrm{d}s-\int\limits_{t}^{\pi} \,\lambda^\star[-\pi, s) \, \mathrm{d}s. 
\end{split}
\end{equation}
From \eqref{eq:Stjielt}, \eqref{eq:antiS} and \eqref{eq:AppFub} it follows that
\begin{equation*}
\begin{split}
\int_{S^1} d_{S^1}(q(t),y) \, \lambda(\mathrm{d}y)=2\int\limits_{t-\pi}^{t}  \,\partial_{-}h(s)\, \mathrm{d}s=4f(q(t))-2\pi C.
\end{split}
\end{equation*}
for all \(t\in [0, \pi)\). Hence, by setting \(\bar{\lambda}=\frac{1}{4}\lambda+C\cdot \Haus^1\), we find that
\begin{equation}\label{eq:result}
\int_{S^1} d_{S^1}(q(t),y) \, \bar{\lambda}(\mathrm{d}y)=f(q(t))
\end{equation}
for all \(t\in [0,\pi)\). Using condition \ref{it:a} and \(\bar{\lambda}(S^1)=C\), a short calculation now shows that \eqref{eq:result} also holds for all \(t\in[-\pi, 0)\).
Hence, \(f=f_{\bar{\lambda}}\), as desired. 

To finish the proof we need to show that \(\lambda\) is unique. To this end,  suppose \(\eta\) is a signed measure such that \(T_\# \eta=-\eta\) (or equivalently \(\eta^a=\eta\)) and \(f\) is representable by \(\bar{\eta}=\eta+C\cdot \Haus^1\). It follows from \eqref{eq:crucial-formula} that 
\[
\lambda[T(x), x)-\lambda[x, T(x))=\eta[T(x), x)-\eta[x, T(x))
\]
for all \(x\in S^1\). Since \(\lambda(S^1)=\eta(S^1)=0\), we obtain that \(\lambda\) and \(\eta\) agree on all half-open intervals \([x, T(x))\). For all \(y\in [x, T(x))\), we have
\[
2\cdot \lambda [x, y)=\lambda[x, T(x))-\lambda[y, T(x)),
\]
and therefore \(\lambda\) and \(\eta\) agree on all half-open intervals of length less than or equal to \(\pi\). Since these intervals generate the Borel \(\sigma\)-algebra of \(S^1\), a standard application of Dynkin’s \(\pi\)-\(\lambda\) theorem yields that \(\lambda=\eta\).
\end{proof}

\begin{proof}[Proof of Lemma~\ref{lem:uno}]
Suppose that \(\lambda\) and \(\eta\) are signed measures such that \(f=f_\lambda=f_\eta\). Since
\[
f(x)+f(-x)=\int_{S^1} \big[d_{S^1}(x, y)+d_{S^1}(-x, y)\big]\, \lambda(\mathrm{d}y)=\pi\cdot \lambda(S^1),
\]
we find that \(\lambda(S^1)=\eta(S^1)\). Now, \eqref{eq:useful-1} tells us that \(f_\lambda=f_{\lambda^a}+\tfrac{\pi}{2}\cdot \lambda(S^1)\) and thereby it follows that \(f_{\lambda^a}=f_{\eta^a}\). Hence, because of Theorem~\ref{thm:main} we have that \(\lambda^a=\eta^a\), as was to be shown. 
\end{proof}

\begin{proof}[Proof of Corollary~\ref{cor:main}]
By a close inspection of the proof of Theorem~\ref{thm:main}  it is readily verified that \(\ref{it:two} \Longrightarrow \ref{it:one}\). Thus, it remains to show that \(\ref{it:one} \Longrightarrow \ref{it:two}\). To this end, suppose \(f\) satisfies all conditions of \(\ref{it:one}\). In particular, \(f\) satisfies conditions \ref{it:a} and \ref{it:b} of Theorem~\ref{thm:main}. Let \(\lambda\) denote the signed measure from Theorem~\ref{thm:main}.
We set \(\mu:=\tfrac{1}{2}\lambda^+\), where \(\lambda=\lambda^+-\lambda^-\) is the Jordan decomposition of \(\lambda\).
By construction, \(T_\#\lambda^+=\lambda^-\) and thus \(\mu(S^1)=\frac{1}{4}\abs{\lambda}(S^1)=\frac{1}{4}\norm{\partial_{-} f}_{_{\Tv(S^1)}}\). Since for all \(x\in S^1\),
\begin{equation*}
\begin{split}
2f_{\bar{\lambda}}(x)-\pi\cdot C&=\int_{S^1} d_{S^1}(x,y) \, \mu(\mathrm{d}y)-\int_{S^1} d_{S^1}(x,T(y)) \, \mu(\mathrm{d}y) \\
&=2 \int_{S^1} d_{S^1}(x,y) \, \mu(\mathrm{d}y)-\pi\cdot\mu(S^1), 
\end{split}
\end{equation*}
it follows that \(f=f_{\bar{\mu}}\), as was to be shown. To finish the proof we need to show that if \(\nu\) is a Borel measure on \(S^1\) with \(\nu(S^1)=\frac{1}{4} \norm{\partial_{-} f}_{_{\Tv(S^1)}}\) such that 
\[
\bar{\nu}=\nu+\big[C-\tfrac{1}{4}\cdot \norm{\partial_{-} f}_{_{\Tv(S^1)}}\big]\cdot\Haus^1
\]
satisfies \(f_{\bar{\nu}}=f_{\bar{\mu}}\), then \(\nu=\mu\). Clearly, \(\bar{\mu}^a=\mu^a\) and thus it follows from Theorem~\ref{thm:main} that \(\mu^a=\nu^a\). In particular,
using that \(\mu(S^1)=\nu(S^1)\), we find that \(\mu[x, T(x))=\nu[x, T(x))\) for all \(x\in S^1\). Now, exactly the same reasoning as in the proof of Theorem~\ref{thm:main} shows that \(\mu=\nu\). This completes the proof.
\end{proof}

We finish this section with an example showing that not every function that is representable by a signed measure is also representable by a non-negative measure.

\begin{example}\label{par:ex} We suppose in the following that the reader is familiar with basic notions from metric geometry. Good general references for this topic are \cite{burago--2001, bridson--1999}. Let \(X\) be the metric space that is obtained by gluing  \(S^1\) and a tripod with edges of length \(\frac{\pi}{3}\) along three equidistant points \(x_1, x_2, x_3\) of \(S^1\). We equip \(X\) with its intrinsic metric \(d\). Let \(o\in X\) denote the center of the tripod. The map \(d_o\colon S^1\to \R\) defined by \(x\mapsto d(x,o)\) satisfies conditions \ref{it:a} and \ref{it:b} of Theorem \ref{thm:main} with \(C=1\). Thus, Theorem \ref{thm:main} tells us that \(d_o\) is representable by a signed measure. Suppose now that there exists a probability measure \(\mu\) on \(S^1\) such that \(f_\mu=d_o\). 
In the following, we show that this is not possible. 

In \cite{creutz-2020}, Creutz showed that \(f_\mu=d_o\) implies that every \(1\)-Lipschitz map from \(S^1\) to a Banach space extends to a \(1\)-Lipschitz map on \(\{o\}\cup S^1\subset X\). However, there exist \(1\)-Lipschitz maps \(\phi\colon S^1\to \R^2\) that do not permit a \(1\)-Lipschitz extension to \(\{o\}\cup S^1\subset X\). 
Indeed, let \(\Delta\subset \R^2\) be an equilateral triangle with perimeter \(2\pi\) and vertices \(v_1, v_2, v_3\) and let \(\phi\colon S^1\to \R^2\) be the inverse map of the map \(\Delta\to S^1\) that is distance-preserving on the edges and sends \(v_i\) to \(x_i\). Clearly, \(\phi\) is \(1\)-Lipschitz. Suppose \(\Phi\colon \{o\}\cup S^1\to \R^2\) is a \(1\)-Lipschitz extension of \(\phi\). Then 
\[
\Phi(o)\in \bigcap_{i=1}^3 B_{\frac{\pi}{3}}(v_i),
\]
which is not possible. Hence, such a map cannot exist, which in turn implies that there is no probability measure \(\mu\) on \(S^1\) with \(f_\mu=d_o\). 
\end{example}
Let \(d_o\) denote the function from Example~\ref{par:ex}. We now give another (shorter) proof of why \(d_o\) cannot be represented 
by any measure.  Since for all \(x\in S^1\), we have that \(d_o(x)+d_o(T(x))=\pi\), it follows that if \(d_o\) were representable 
by a measure \(\mu\), then \(\mu\) must necessarily be a probability measure. But a straightforward calculation reveals that
\[
\norm{\partial_{-} d_o}_{_{\Tv(S^1)}}=12>4;
\]
hence, Corollary~\ref{cor:main} implies that \(d_o\) is not representable by a probability measure. 
Thus, it is not representable by any measure.  

\let\oldbibliography\thebibliography
\renewcommand{\thebibliography}[1]{\oldbibliography{#1}
\setlength{\itemsep}{3pt}}

\bibliography{refs}
\bibliographystyle{alpha}

\end{document}